\documentclass[a4paper,12pt,reqno]{amsart}
\usepackage{amsmath}
\usepackage{graphicx} 
\usepackage{amssymb} 
\usepackage{tikz}
\usetikzlibrary{cd}
\usetikzlibrary{arrows.meta}
\usepackage{geometry}
\usepackage{xcolor}
\usepackage{float}
\usepackage{hyperref}
\usepackage{esint}
\usepackage{subcaption}
\usepackage{ upgreek }
\usepackage[backend=biber,style=ieee,sorting=none]{biblatex}
\usepackage{mathrsfs}
\usepackage{booktabs}
\usepackage{array}
\usepackage{seqsplit}

\newcommand{\dashedhrule}{%
  \noindent\tikz{\draw[dash pattern=on 4pt off 3pt, line width=0.4pt] (0,0) -- (\linewidth,0);}%
}

\newcommand{\constentry}[4]{%
  \noindent $p = #1$\par
  \noindent $N = #2$\par
  \noindent $\mathrm{const}(p) = $\seqsplit{#3}\par
  \noindent\textbf{Prime factorization:} #4\par
  \vspace{0.5em}
  \dashedhrule
  \vspace{0.5em}
}
\newcommand{\dentry}[3]{%
  \noindent $p = #1$\par
    \noindent $N = #2$\par
      \noindent $D(p) = $\seqsplit{#3}\par
        \vspace{0.5em}
          \dashedhrule
            \vspace{0.5em}
            }
\hypersetup{
    colorlinks=true,
    linkcolor=blue,     
    citecolor=blue,     
    urlcolor=blue       
}
\newtheorem{theorem}{Theorem}[section]
\newtheorem{corollary}{Corollary}[theorem]
\newtheorem{lemma}[theorem]{Lemma}
\theoremstyle{definition}
\newtheorem*{convention}{Convention}
\newtheorem{remark}{Remark}
\newtheorem{conjecture}{Conjecture}
\newtheorem{notation}{Notation}
\newtheorem{example}{Example}[section]
\newtheorem{claim}{Claim}[section]
\DeclareMathOperator\Wronsk{Wronsk}

\title[Wronskian Identity for Iterated Differential Operators]
{The alternating compositions of weighted differential operators yield The weights' Wronskian with which constant\,?}
\author[K.\,C.\,Shah, A.\,V.\,Kiselev]{%
  Kian C.\ Shah
  \and
  Arthemy V.\ Kiselev
}\thanks{\textit{Address:} Bernoulli Institute for Mathematics,
  Computer Science and Artificial Intelligence, University of Groningen,
  P.O.\ Box 407, 9700\,AK Groningen, The Netherlands.}\thanks{}
  \dedicatory{\textup{Based on the talk by the first author at the \textsc{Group36} colloquium on group theoretic methods in Physics (13--17 July 2026, Valladolid, Spain) and poster presentation at the 17th Algorithmic Number Theory symposium (6--10 July 2026, Groningen, NL), and on the talks by the second author at the Prague Mathematical Physics Seminar in the Charles University (7 May 2026, Prague, CZ), at the Mathematics seminar in the IH\'ES (13 May 2026, Bures\/-\/sur\/-\/Yvette, France), and at the Open Problems session within the ISNMP 2nd conference on Nonlinear Mathematical Physics (28 June -- 4 July 2026, Bad Ems, Germany).}}


\date{\today}
\begin{document}
\begin{abstract}
The alternated composition of $N=2p$ differential operators $w_j(x)\,\partial_x^p$ of strict order~$p$ on the line $\mathbb{R}\ni x$ is again a differential operator of strict order~$p$; its coefficient is the constant $\mathrm{const}(p)$, depending only on the arity~$N$, times the Wronskian determinant of the originally taken coefficients $w_1,\dots,w_N$. The case $p=1$ of the Lie bracket for two vector fields fixes $\mathrm{const}(1)=1$, and $\mathrm{const}(2)=2$ is found easily by hand; $\mathrm{const}(3)=90$ can still be obtained symbolically. The problem is to determine $\mathrm{const}(p\geqslant4)$.

We compute $\mathrm{const}(p)$ exactly for all $p\leqslant14$ -- a 241-digit integer at $p=14$ -- and record the resulting integer sequence as OEIS A392714. We prove that $v_p(\mathrm{const}(p))\geqslant p-1$ for every prime $p$, matching the exact equality observed numerically throughout our range, and conjecture that this equality holds in general.

We show that $\log\mathrm{const}(p)$ grows like $\alpha p^2\log p$, with the leading coefficient close to $2$, nearly saturating the bound $p^2\log p\,(1+O(1/\log p))\leqslant\log\mathrm{const}(p)\leqslant2p^2\log p\,(1+O(1/\log p))$ obtained by O.~Zaboronski (private communication). A naturally arising reduced constant is found to decay to zero super-exponentially rather than grow, a direct consequence of this near-saturation.
\end{abstract}
\maketitle{}

\section{Introduction}
\label{sec:intro}
\noindent Let $p \in \mathbb{Z}_{\geq 1}$, and set $N = 2p$; let $w_1, \dots, w_N$ be the coefficients of differential operators of strict order \(p\) on the interval $I \subseteq \mathbb{R}$, with coordinate \(x\). For any smooth function \(f\in C^\infty(\mathbb{R}) \), define the \emph{alternating composition} of these differential operators:
\begin{equation}\label{eq:opdef}
  \mathcal{A}_p[w_1,\dots,w_N](f) \;:=\; 
  \sum_{\sigma \in S_N} (-1)^\sigma\, 
  w_{\sigma(1)}\partial_x^p \circ w_{\sigma(2)}\partial_x^p 
  \circ \cdots \circ w_{\sigma(N)}\partial_x^p\,(f(x)),
\end{equation}
where the sum runs over all permutations $\sigma$ of the set $\{1,\dots,N\}$ and $(-1)^\sigma$ denotes the sign of~$\sigma$. 
\begin{theorem}[{\cite{Kiselev2005AssociativeWronskians},\cite{Kiselev_2025}}]\label{thm:main}
For any smooth function $f$ and weight functions $w_1, \dots, w_N$, 
\begin{equation}
    \label{eq:main}
    \mathcal{A}_p[w_1,\dots,w_N](f) 
  \;=\; \mathrm{const}(p) \cdot 
  \mathrm{Wronsk}(w_1, \dots, w_N) \cdot \partial_x^p(f(x)),
\end{equation}
where $\mathrm{const}(p)$ is a universal constant, independent of $f$ and the weight functions.
\end{theorem}
The proof of this theorem has been outlined in \autoref{appendixa}. The problem which we study in this paper is the calculation of the constant's values~$\mathrm{const}(p)$ in identity~\eqref{eq:main}. 
\begin{lemma}\label{lem:order}
    In every non-vanishing term of \eqref{eq:opdef}, the function 
    $f$ appears differentiated to precisely order $p$. Consequently, 
    $\partial_x^p(f)$ factors out of the entire alternating sum, and 
    \eqref{eq:main} reduces to the following equation:
    \begin{equation}
    \label{eqn:reduced}
    \sum_{\sigma \in S_N} (-1)^\sigma\, w_{\sigma(1)}\partial_x^p \circ w_{\sigma(2)}\partial_x^p 
  \circ \cdots \circ w_{\sigma(N)}\;=\; \mathrm{const}(p) \cdot \mathrm{Wronsk}(w_1, \dots, w_N).
\end{equation}
\end{lemma}
For small values \(p\leq3\) we verify identity~\eqref{eq:main} by hand for \(p=1\), whence \(\mathrm{const}(1)=1\), and for \(p=2\), whence we get \(\mathrm{const}(2)=2\) (c.f.\ Example~\ref{ex:peq2} on page~\pageref{ex:peq2}). For \(p=3\), using Jets (a package for differential calculus, see \cite{BaranMarvan_Jets},\cite{Marvan2009}) in Maple, we expand the left-hand side of Equation~\eqref{eq:main} by taking the entire sum over \(6!=120\) permutations (here \(N=2\cdot3=6\) and all differential operators have order \(p=3\)); we find \(\mathrm{const}(3)=90\). The values of $\mathrm{const}(p)$ for small $p$ are summarised in Table~\ref{tab:intro}.
\begin{table}[htb]
\caption{Previously known values of $\mathrm{const}(p)$ for $p=1,2,3$. 
The case $p=1$ recovers the Lie bracket of two vector fields. 
Extended numerical data for $p \leq 6$, including the size and 
parity structure of the contributing permutation set $\Phi_p$, 
are recorded in Table~\ref{tab:constants} at the end of the paper.}
\label{tab:intro}
\centering
\renewcommand{\arraystretch}{1.4}
\begin{tabular}{cc l}
\hline\hline
$p$ \phantom{MMM} 
& $N!$\phantom{MMM} & $\mathrm{const}(p)$ \\
\hline
$1$ \phantom{MMM} 
& $2$\phantom{MMM}  & $1$\quad (Lie)  \\
$2$ \phantom{MMM} 
& $24$ \phantom{MMM} 
& $2$  \\

$3$ \phantom{MMM} 
& $720$ \phantom{MMM} 
& $90$ \\
\hline\hline
\end{tabular}
\end{table}

However simple be the definition of~\(\mathrm{const}(p)\), its calculation at \(p=4\) becomes unfeasible using Jets if one still tries a straightforward expansion of both sides in \eqref{eq:main} or \eqref{eqn:reduced}, then solving for the proportionality constant. To learn~\(\mathrm{const}(p\geq4)\), we need a new, more economical strategy.

This paper is structured as follows. In Section~\ref{sec:monomials} we choose the weight functions $w_{1},\dots,w_{N}$ to be monomials in $x$ such that the r.-h.s.\ of \eqref{eqn:reduced} becomes the product of still unknown $\mathrm{const}(p)$ and an explicitly known integer number.  
Simultaneously, our choice of the weight functions reduces the number of permutations (from $S_N$) actually contributing to the l.-h.s.\ of \eqref{eqn:reduced}; in Section~\ref{sec:algorithm} we describe the filtering and generative algorithms to build the set $\Phi_p \subsetneq S_{N=2p}$ of permutations that do contribute. Now, in Section~\ref{sec:closedform} we derive the formula of $\mathrm{const}(p)$; it amounts to the summation only over the set $\Phi_p$. We apply this formula to reach the exact values $\mathrm{const}(p \leqslant 14 )$; we present these values and more information about the sets $\Phi_p$ of `late-growing' permutations in Table \ref{tab:constants} in Section~\ref{sec:results}.

\begin{convention}
Throughout this paper, we adopt the one-line notation for permutations, in which $\sigma = (\sigma(1), \sigma(2), \dots, \sigma(N))$ records the image of each position. That is, $\sigma(k)$ denotes the index of the object occupying position $k$ after the permutation is applied. Under this convention, the identity element of $S_N$ is represented as $\sigma = (1, 2, \dots, N)$. For instance, in $S_4$, the identity $\sigma = (1,2,3,4)$ corresponds to the map
$$
(1,\, x,\, x^2,\, x^3) \;\longmapsto\; (1,\, x,\, x^2,\, x^3),
$$
while $\sigma = (4,1,2,3)$ corresponds to
$$
(1,\, x,\, x^2,\, x^3) \;\longmapsto\; (x^3,\, 1,\, x,\, x^2).
$$
\end{convention}
\section{Combinatorics of the monomial case}
\label{sec:monomials}
\noindent To make the r.h.s of Equation~\eqref{eqn:reduced} shorter and extract the constants $\mathrm{const}(p)$ explicitly, we choose specific weight functions, namely, let them be monomials:
\[
  w_n(x) = x^{n-1}, \qquad n = 1, \dots, N.
\]
\begin{remark}
    Note that $\mathrm{deg}$$(w_{\sigma(k)})=\sigma(k)-1$ whenever \(w_k=x^{k-1}\).
\end{remark}
\begin{lemma}\label{lem:wronskian}
  The Wronskian of $1, x, x^2, \dots, x^{N-1}$ equals
  \[
    \mathrm{Wronsk}(1,\, x,\, x^2,\, \dots,\, x^{N-1}) 
    \;=\; \prod_{k=0}^{N-1} k!\,.
  \]
\end{lemma}

\begin{proof}
We note that the Wronskian matrix of $1, x, \dots, x^{N-1}$ is upper triangular, with diagonal entries $0!, 1!, 2!, \dots, (N-1)!$. Thus, the determinant is the product of these entries.
\end{proof}

To understand the action of the iterated operators in \eqref{eq:opdef} in the monomial case, consider two monomials $x^\alpha$ and $x^\beta$ with $\alpha,\beta\in\mathbb{Z}_{\ge 0}$. The operator $x^\beta\partial_x^p$ acts on $x^\alpha$: whenever $\alpha<p$ the result is zero, otherwise
\begin{equation}
    \label{eqn:opaction}
    x^\beta\,\partial_x^p\,(x^\alpha) 
    \;=\; x^\beta\cdot\frac{\alpha!}{(\alpha-p)!}\, x^{\alpha-p}\;=\; \frac{\alpha!}{(\alpha-p)!}\, x^{\alpha+\beta-p}.
\end{equation}
Each (when nonzero in the output) application of the weighted operator $x^\beta\partial_x^p$ multiplies the argument $x^\alpha$ by the falling factorial; the notation is
\[
\alpha^{\underline{p}}=\frac{\alpha!}{(\alpha-p)!}=\alpha(\alpha-1)\cdots(\alpha-p+1);
\]
also, the operator shifts the exponent by $\alpha\mapsto\alpha+\beta-p$.
\begin{remark}
The vanishing condition $\alpha < p$ immediately implies that not every permutation $\sigma \in S_N$ contributes to the alternating sum in \eqref{eq:main}. Any permutation for which the innermost weight function $w_{\sigma(N)} = x^{\sigma(N)-1}$ satisfies $\sigma(N) - 1 < p$ yields zero contribution; clearly, such a permutation may be discarded at the outset.
\end{remark}

\begin{example}
Let $p = 2$ and $N = 4$, so that the weight functions are 
$\{1,x,x^2,x^3\}$. Consider the permutation $\sigma = (3,1,4,2)$ (corresponding to \((1,\, x,\, x^2,\, x^3) \;\mapsto\; (x^2,\, 1,\, x^3,\, x)\)), so that $\sigma(4) = 2$ and the innermost operator acts as $\partial_x^2(w_2) = \partial_x^2(x) = 0$. The entire term arising from this permutation thus vanishes, and $\sigma$ does not contribute to the sum. More generally, any permutation with $\sigma(N) \leq p$ is eliminated by this filter; for $p=2, N=4$ this removes all $\sigma$ with $\sigma(4) \in \{1, 2\}$, 
reducing the number of permutations to be evaluated from $4! = 24$ to those with $\sigma(4) \in \{3,4\}$.
\end{example}
As illustrated by the above example, we can now shift our focus from the monomial weight functions to the permutations themselves; from now on, we work with permutations directly.
\begin{lemma}\label{lem:sigma1}
    For any contributing permutation $\sigma \in S_N$, we have $\sigma(1) = 1$, that is, the weight function $w_{\sigma(1)} = w_1 = 1$ is necessarily fixed in the leading leftmost position.
\end{lemma}
\begin{proof}
    There are two kinds of permutations: the ones which keep the leftmost monomial \(w_1=1\) in its first place, and those which send it away, so that a positive-degree monomial \(x^{\alpha}\) becomes the leftmost. Suppose the latter, so that it is this positive-degree monomial \(x^{\alpha}\) which itself is not differentiated (but is multiplied by a nonnegative-degree monomial resulting from all the differentiations which stood to the right of it). This yields a contribution of degree~\(\ge \alpha\) to the sum. But we know that in the r.-h.s.\ of Equation~\eqref{eqn:reduced}, the Wronskian of \(1,x,\ldots,x^{N-1}\) is a real number, i.e.\ is of polynomial degree zero. Therefore, every permutation sending \(1\) away cannot contribute to the l.-h.s.\ of \eqref{eqn:reduced}.
\end{proof}
To identify all the permutations (from the full set \(S_N\) of size \(N!\)) that do contribute to the sum in the l.-h.s.\ of \eqref{eqn:reduced}, we build an algorithm; the filter combines Lemma~\ref{lem:sigma1} and formula \eqref{eqn:opaction}.
\section{Algorithm to filter permutations and a faster generative procedure}
\label{sec:algorithm}
\noindent\textbf{Generator (ver. 1).} We start by generating the set of all $N!$ permutations of ${1, 2, \dots, N}$ by a standard recursive construction. Then, any permutation with $\sigma(1) \neq 1$ is discarded. 
The remaining permutations are tested by simulating the exponent-tracking process from right to left. Define the partial sum $$ T_k \;:=\; \sum_{i=N-k+1}^{N} \bigl(\sigma(i) - 1-p\bigr), \qquad k = 1, \dots, N-1. $$ This quantity tracks whether the exponent of $x$ remains non-negative after each successive operator application. If $T_k < 0$ for some $k$, the exponent would become negative at that stage, causing the term to vanish; this permutation is therefore rejected. Only those permutations for which $T_k \geq 0$ for all $k$ are retained. 
\begin{notation}
    We denote by \(\Phi_p\) the set of all the remaining permutations (from \(S_N\)) which pass the filter and contribute to the sum in the l.-h.s.\ of \eqref{eqn:reduced}; note \(\Phi_p \subsetneq S_N\).
\end{notation}

\noindent\textbf{Generator (ver. 2).}The same idea as above can be used to make a generative procedure, which avoids generating invalid permutations in the first place. Rather than constructing all $N!$ permutations \textit{ab initio} and filtering \textit{post hoc}, it builds permutations incrementally via backtracking, enforcing the non-negativity condition $T_k \geq 0$ at each step of the construction.

Concretely, the element $\sigma(1) = 1$ is fixed from the beginning. The remaining positions are filled from right to left: at each stage, a candidate element $R$ is appended to the current partial permutation only if the updated running sum (\verb|current_sum += (R-(p+1))|) remains non-negative. If it does not, that branch of the search tree is removed entirely, and no permutations extending that partial assignment are ever constructed. When the remaining pool of numbers $\{2,3,\cdots,N\}$ is exhausted, the completed suffix is prepended with $1$ to yield a valid permutation.

The two algorithms are equivalent in output, as they both implement the same pair of conditions. However, the approach in version 2 prunes the search tree early and avoids the $O(N!)$ memory cost of full enumeration, unlike the approach in version 1. For the purposes of obtaining the set \(\Phi_p\) (and eventually computing $\text{const}(p)$), version 2 of the algorithm is strictly preferable, and is what makes the computation possible up to $p = 5$ on modest hardware.

\begin{example}\label{ex:p1p2}
We illustrate the filtering procedure for $p=1$ and $p=2$.

\textbf{Case $p=1$, $N=2$.} The group $S_2$ consists 
of two permutations, and thus the alternating sum \eqref{eqn:reduced} reads
\[
    \sum_{\sigma \in S_2} (-1)^\sigma\, w_{\sigma(1)}\partial_x \circ 
    w_{\sigma(2)}.
\]
The two terms are as follows. For $\sigma = (1,2)$, we obtain the term $1 \cdot \partial_x(x)$, which does contribute. For $\sigma = (2,1)$, we obtain the term $x \cdot \partial_x(1)$, which does not contribute because \(\partial_x(1)=0\). This is also consistent with Lemma~\ref{lem:sigma1}, since $\sigma(1) = 2 \neq 1$, this permutation is discarded. Thus
\[
    \Phi_1 = \{(1,2)\} \subset S_2, \qquad 
    |\Phi_1| = 1 = \tfrac{1}{2}\,|S_2|.
\]

\textbf{Case $p=2$, $N=4$.} The group $S_4$ has $4! = 24$ elements. Applying Lemma~\ref{lem:sigma1} immediately reduces it to the $3! = 6$ permutations with $\sigma(1)=1$. Of these, the non-negativity condition on the exponents $T_k$ eliminates three more, leaving the valid set
\[
    \Phi_2 = \{(1,2,3,4),\ (1,2,4,3),\ (1,3,2,4)\} \subset S_4, \qquad|\Phi_2| = 3 = \tfrac{1}{8}\,|S_4|.
\]
The three discarded permutations with $\sigma(1)=1$ are $(1,3,4,2)$, $(1,4,2,3)$, and $(1,4,3,2)$, each of them produces a negative exponent at some intermediate stage and therefore contributes a zero term to the sum.
\end{example}
\begin{remark}
    The permutations that are `close to' the identity permutation $(1, 2, \dots, N)$ contribute to the sum, while permutations that are `close to' the reversed order permutation $(N, N-1, \dots, 2, 1)$ do not contribute to the sum. Indeed, permutations near the identity present the weight functions in increasing order of their degree, ensuring that each intermediate exponent remains sufficiently large to survive $p$ differentiations. On the other hand, permutations near the reversed order present lower-degree weights in the innermost operators, which has a higher chance of driving the exponent below zero.
\end{remark}
\section{A closed-form expression for \texorpdfstring{$\mathrm{const}(p)$}{const(p)}}
\label{sec:closedform}
\noindent In the previous section, we reduced the alternating sum to the summation over the filtered set $\Phi_p$, that is,
\[
\sum_{\sigma \in S_N} (-1)^\sigma\, w_{\sigma(1)}\partial_x^p \circ 
w_{\sigma(2)}\partial_x^p \circ \cdots \circ w_{\sigma(N)}
\;=\;
\sum_{\sigma \in \Phi_p \subset S_N} (-1)^\sigma\, 
w_{\sigma(1)}\partial_x^p \circ w_{\sigma(2)}\partial_x^p 
\circ \cdots \circ w_{\sigma(N)};
\]
let us emphasize that the equality holds for the special choice of monomial weights~\(
w_1 = 1, \dots, w_N = x^{N-1}
\); the set \(\Phi_p\) was constructed in reference to this choice. We now derive a closed-form expression for each term in this sum, yielding an explicit formula for $\mathrm{const}(p)$.
\begin{notation}
    For any integer \(k\) from the range \(1\le k\le N-1\), and a permutation \(\sigma\in\Phi_p\), we, by definition, put
    \begin{equation}\label{eq:Ek}
        E_k \;:=\; \sum_{j=0}^{k-1}\bigl(\sigma(N-j)-1\bigr) - (k-1)p,
    \end{equation}
    so that its falling factorial is
    \[
    E_k^{\,\underline{p}}= \frac{E_k!}{(E_k - p)!}.
    \]
\end{notation}
\begin{claim}\label{claim:closedform}
    \begin{equation}\label{eq:closedform}
        \mathrm{const}(p) \cdot \prod_{k=0}^{N-1} k! 
        \;=\; 
        \sum_{\sigma \in \Phi_p} (-1)^\sigma \prod_{k=1}^{N-1} 
        E_k^{\,\underline{p}}
    \end{equation}
\end{claim}

\begin{proof}
For our chosen set of monomial weights, each surviving term in the l.-h.s.\ of \eqref{eqn:reduced} is
\begin{multline*}
    \left(1\cdot\partial_x^p\circ w_{\sigma(2)}\cdot\partial_x^p\circ\dots\circ w_{\sigma(N-1)}\cdot\partial_x^p\right)(w_{\sigma(N)})\\=\left(1\cdot\partial_x^p\circ x^{\sigma(2)-1}\cdot\partial_x^p\circ\dots\circ x^{\sigma(N-1)-1}\cdot\partial_x^p\right)(x^{\sigma(N)-1}).
\end{multline*}
The r.h.s\ of formula~\eqref{eq:closedform} tracks the scalar coefficient accumulated as the composition of weighted operators acts on the monomial $x^{\sigma(N)-1}$, proceeding from right to left through the $N-1$ applications of $\partial_x^p$. Let us study this right-to-left action of differential operators in more detail.

\noindent\textbf{Rightmost action.} The rightmost operator acts on the weight $w_{\sigma(N)} = x^{\sigma(N)-1}$:
\[
    \partial_x^p\!\left(x^{\sigma(N)-1}\right) 
    = \frac{(\sigma(N)-1)!}{(\sigma(N)-1-p)!}\, x^{\sigma(N)-1-p}
    = E_1^{\,\underline{p}}\cdot x^{\sigma(N)-1-p},
\]
where $E_1 = \sigma(N)-1$. The exponent of $x$ is now $\sigma(N)-1-p$. 

\noindent\textbf{Second action.} Pre-multiplying by $w_{\sigma(N-1)} = x^{\sigma(N-1)-1}$ shifts the exponent to
\[
    (\sigma(N-1)-1) + (\sigma(N)-1-p) 
    = \sigma(N-1)+\sigma(N)-2-p = E_2,
\]
and a further application of $\partial_x^p$ yields the factor
\[
    \frac{E_2!}{(E_2-p)!} = E_2^{\,\underline{p}}.
\]
\textbf{General step.} After $k$ actions of $\partial_x^p$, the accumulated exponent of $x$ is $E_k - p$. Pre-multiplying by $w_{\sigma(N-k)} = x^{\sigma(N-k)-1}$ raises the exponent to $E_{k+1}$, and the $(k+1)$-th application of $\partial_x^p$ contributes the factor $E_{k+1}^{\,\underline{p}}$. Iterating through all the $N-1$ actions, the total scalar coefficient accumulated for the permutation $\sigma\in\Phi_p$ is
\[
    c(\sigma,p) \;=\; \prod_{k=1}^{N-1} E_k^{\,\underline{p}}.
\]
Summing over all the contributing permutations with their signs and equating it to the product of const($p$) and  the Wronskian determinant $\mathrm{Wronsk}(1,x,\dots,x^{N-1}) = \prod_{k=0}^{N-1}k!$ from Lemma~\ref{lem:wronskian} yields equality \eqref{eq:closedform}.
\end{proof}

\begin{example}
\label{ex:peq2}
We now illustrate Claim~\ref{claim:closedform} for $p=2$, $N=4$. The contributing permutations, as identified in Example~\ref{ex:p1p2}, are $$\Phi_2 = \{(1,2,3,4),\, (1,2,4,3),\, (1,3,2,4)\}.$$ We compute the coefficient $c(\sigma, 2) = \prod_{k=1}^{3} E_k^{\,\underline{2}}$ for each filtered permutation \(\sigma\in\Phi_2\):

\medskip
For $\sigma = (1,2,3,4)$, the sign is $(-1)^\sigma = +1$ and
\[
E_1 = 3,\quad E_2 = 3+2-2 = 3,\quad E_3 = 3+1-2 = 2,
\]
\[
c((1,2,3,4),\,2) = 3^{\underline{2}}\cdot 3^{\underline{2}}
\cdot 2^{\underline{2}} = 3!\cdot 3!\cdot 2!.
\]

For $\sigma = (1,2,4,3)$, the sign is $(-1)^\sigma = -1$ and
\[
E_1 = 2,\quad E_2 = 2+3-2 = 3,\quad E_3 = 3+1-2 = 2,
\]
\[
c((1,2,4,3),\,2) = 2^{\underline{2}}\cdot 3^{\underline{2}}
\cdot 2^{\underline{2}} = 3!\cdot 2!\cdot 2!.
\]

For $\sigma = (1,3,2,4)$, the sign is $(-1)^\sigma = -1$ and
\[
E_1 = 3,\quad E_2 = 3+1-2 = 2,\quad E_3 = 2+2-2 = 2,
\]
\[
c((1,3,2,4),\,2) = 3^{\underline{2}}\cdot 2^{\underline{2}}
\cdot 2^{\underline{2}} = 3!\cdot 2!\cdot 2!.
\]

\medskip
\noindent Assembling the signed sum and dividing by the Wronskian determinant $\mathrm{Wronsk}(1,x,x^2,x^3) = 3!\cdot 2!\cdot 1!$ from Lemma~\ref{lem:wronskian}, we obtain that

\begin{align*}
    \mathrm{const}(2) 
    &= \frac{\overbrace{3!\cdot 3!\cdot 2!}^{c(\sigma_1,\,2)} 
    - \overbrace{3!\cdot 2!\cdot 2!}^{c(\sigma_2,\,2)} 
    - \overbrace{3!\cdot 2!\cdot 2!}^{c(\sigma_3,\,2)}}
    {\underbrace{3!\cdot 2!\cdot 1!}_{\mathrm{Wronsk}}}
    = 3! - 2! - 2!\\[6pt]
    &= 6 - 2 - 2 \;=\; 2,
\end{align*}
in agreement with the value \(\mathrm{const}(2)=2\) recorded in Table~\ref{tab:intro}.
\end{example}
\section{An intermediate acceleration: the deviation-vector formulation}
\label{sec:deviation}
\noindent The filtering procedure of Section~\ref{sec:algorithm} associates to each permutation a binary outcome: either the permutation belongs to \(\Phi_p\) or it does not. The following deviation-vector formulation refines this viewpoint by assigning to each contributing permutation a vector of nonnegative integers; in doing so we measure its distance from the extremal case given by the identity permutation. In this sense, it builds on our qualitative condition by also measuring how strongly a contributing permutation \(\sigma\in\Phi_p\) satisfies the late-growing condition. 

The motivation to introduce this statistic on late-growing permutations and provide the reformulation below was to develop an improvement of the original algorithm. Although it yields a significant speedup and was sufficient to obtain \(\mathrm{const}(7)\), it was later superseded by a fully recursive method with memoization. We nevertheless present it because we now obtain a mathematically whole idea of the role of late-growing permutations in our computations.

\medskip
\noindent\textbf{The master vector.}
Let $\sigma_{\mathrm{id}} = (1,2,\ldots,N)$ denote the identity permutation, which always
belongs to $\Phi_p$. Substituting $\sigma_{\mathrm{id}}$ into~\eqref{eq:Ek}, the $k$-th
component of the resulting vector is
\[
    v^{\mathrm{m}}_k 
    \;:=\; E_k\!\left(\sigma_{\mathrm{id}}\right)
    \;=\; \sum_{j=0}^{k-1}(N - 1 - j) \;-\; (k-1)p
    \;=\; k(N-1) - \tfrac{k(k-1)}{2} - (k-1)p.
\]
We call $\mathbf{v}^{\mathrm{m}} = (v^{\mathrm{m}}_1, \ldots, v^{\mathrm{m}}_{N-1})$
the \emph{master vector}; it is determined solely by $p$.

\medskip
\noindent\textbf{The deviation vector.}
For any $\sigma\in\Phi_p$, define the \emph{deviation vector}
$\boldsymbol{\delta}(\sigma) = (\delta_1(\sigma),\ldots,\delta_{N-1}(\sigma))$ by
\[
    \delta_k(\sigma) 
    \;:=\; v^{\mathrm{m}}_k - E_k(\sigma)
    \;=\; \sum_{j=0}^{k-1}\bigl[(N - j) - \sigma(N-j)\bigr].
\]
Each $\delta_k(\sigma) \geq 0$ tracks how much the right-to-left running sum of $\sigma$
falls short of the identity. In particular, $\delta_k = 0$ for all $k$ if and only if
$\sigma = \sigma_{\mathrm{id}}$.

\medskip
\noindent\textbf{The product function.}
Define $F\colon \mathbb{Z}_{\geq p}^{N-1} \to \mathbb{Z}_{>0}$ by
\[
    F(\mathbf{v}) 
    \;:=\; \prod_{k=1}^{N-1} \frac{v_k!}{(v_k - p)!}
    \;=\; \prod_{k=1}^{N-1} v_k^{\,\underline{p}}.
\]

\medskip
\noindent\textbf{Reformulated claim.}
Formula~\eqref{eq:closedform} can now be written as
\begin{equation}\label{eq:deviation}
    \mathrm{const}(p) \cdot \prod_{k=0}^{N-1} k!
    \;=\;
    \sum_{\sigma\in\Phi_p} (-1)^\sigma\,
    F\!\left(\mathbf{v}^{\mathrm{m}} - \boldsymbol{\delta}(\sigma)\right).
\end{equation}
The filtering condition $T_k \geq 0$ (which defines $\Phi_p$) is precisely the condition
that $E_k(\sigma) \geq p$ for all $k$, i.e.\ $\delta_k(\sigma) \leq v^{\mathrm{m}}_k - p$.
The filter thus guarantees that all arguments of $F$ remain in the domain $\mathbb{Z}_{\geq p}$.

After this reformulation, the entire computation of $\mathrm{const}(p)$ reduces to: generate $\Phi_p$ incrementally, read off $\boldsymbol{\delta}(\sigma)$ for free during generation, evaluate $F$,
and accumulate the signed sum.
\begin{example}
  Consider $p=3$, $N=6$; we have that
\[
    \mathbf{v}^{\mathrm{m}}
    =
    (5,6,6,5,3).
\]
This vector corresponds exactly to the factorial arguments appearing in the leading (identity) term of the expansion:
\[
    +\;
    \frac{5!}{2!}
    \cdot
    \frac{6!}{3!}
    \cdot
    \frac{6!}{3!}
    \cdot
    \frac{5!}{2!}
    \cdot
    \frac{3!}{0!},
\]
associated to the permutation
\[   
    (1,2,3,4,5,6).
\]
\noindent Naturally, every admissible permutation contributes a term obtained by lowering certain entries of the master vector. For instance,
\[
    (1,2,3,4,6,5)
\]
produces
\[
    (4,6,6,5,3),
\]
and hence contributes
\[
    -\;
    \frac{4!}{1!}
    \cdot
    \frac{6!}{3!}
    \cdot
    \frac{6!}{3!}
    \cdot
    \frac{5!}{2!}
    \cdot
    \frac{3!}{0!}.
\]
As a final example,
\[
    (1,2,3,5,4,6)
\]
produces
\[
    (5,5,6,5,3),
\]
giving the contribution
\[
    -\;
    \frac{5!}{2!}
    \cdot
    \frac{5!}{2!}
    \cdot
    \frac{6!}{3!}
    \cdot
    \frac{5!}{2!}
    \cdot
    \frac{3!}{0!}.
\]
Thus we have that the corresponding deviation vectors of our examples above (except for the identity permutation) are:
\[
    \boldsymbol{\delta}(1,2,3,4,6,5)
    =
    (1,0,0,0,0),
\]
and
\[
    \boldsymbol{\delta}(1,2,3,5,4,6)
    =
    (0,1,0,0,0).
\]

\end{example}
\begin{remark}
    For small values of $p$, the `master' vectors are:
\[
p=1,\quad N=2:
\qquad
\mathbf v^{\mathrm m}=(1).
\]
\[
p=2,\quad N=4:
\qquad
\mathbf v^{\mathrm m}=(3,3,2).
\]
\[
p=3,\quad N=6:
\qquad
\mathbf v^{\mathrm m}=(5,6,6,5,3).
\]
\[
p=4,\quad N=8:
\qquad
\mathbf v^{\mathrm m}=(7,9,10,10,9,7,4).
\]
\end{remark}
\section{Initial numerical results}
\label{sec:results}
\noindent We have reduced the problem of finding the values of $\mathrm{const}(p)$ in~\eqref{eq:main} or~\eqref{eqn:reduced} to the construction of the (sub)set of permutations $\Phi_{p}\subseteq S_{N}$ (here $N= 2p$) and then, to a rather straightforward calculation of the sought value $\operatorname{const}(p)$ by using formula~\eqref{eq:closedform} with the summation over $\Phi_{p}$ (see Table~\ref{tab:constants} and Appendix~\ref{sec:allvalues}).
\begin{table}[H]
\centering
\caption{For each $p \leqslant 6$, we report the total number of permutations $N! = (2p)!$, 
the size of the contributing set $|\Phi_p|$, its fraction of the size of the group $S_N$, the number of even and odd permutations in $\Phi_p$, and the 
corresponding value of $\mathrm{const}(p)$. The values $\mathrm{const}(p \leqslant 14)$ are in Appendix~\ref{sec:allvalues}.}
\renewcommand{\arraystretch}{1.4}
\begin{tabular}{c r r r r r r}
\hline\hline
$p$ & $N!$ & $|\Phi_p|$ & $|\Phi_p|/N!$ & even & odd & $\mathrm{const}(p)$ \\
\hline
$1$ & $2$ & $1$ & $1/2$ & $1$ & $0$ & $1$ \\
$2$ & $24$ & $3$ & $1/8$ & $1$ & $2$ & $2$ \\
$3$ & $720$ & $35$ & $\approx1/20$ & $18$ & $17$ & $90$ \\
$4$ & $40\,320$ & $1\,001$ & $\approx1/40$ & $500$ & $501$ & $586\,656$ \\
$5$ & $3\,628\,800$ & $53\,109$ & $\approx1/68$ & $26\,555$ & $26\,554$ & $1\,915\,103\,977\,500$ \\
$6$ & $479\,001\,600$ & $4\,605\,271$ & $\approx1/104$ & $2\,302\,635$ & $2\,302\,636$ & $\approx 7.886 \times 10^{21}$ \\
\hline\hline
\end{tabular}

\label{tab:constants}
\end{table}
We discover that, as $p$ grows (compared with the small values $p=1,2,3$ in Table~\ref{tab:intro}), our strategy to restrict the summation to the subset $\Phi_{p}\subseteq S_{N}$ becomes more and more effective. In the third and fourth columns of Table~\ref{tab:constants} we present the size of the set $\Phi_{p}$ (for $p\le 6$) and its fraction relative to $N!=$(the size of the group $S_N$); we see that the denominators in the approximations $|\Phi_{p}|/N!$ grow monotonically.
\begin{remark}
\label{rem:oddevenconj}
    We note one striking structural feature of this data, concerning the distribution of even and odd permutations within \(\Phi_p\). Namely, the counts of even and odd permutations in \(\Phi_p\) always differ by exactly 1, with the dominant parity alternating with \(p\): odd permutations dominate for even \(p\) and vice versa. We record this as a conjecture for \(p\ge7\) onwards.
\end{remark}
The sequence $|\Phi_p| = 1, 3, 35, 1001, 53109, 4605271, \dots$ for $p = 1, 2, 3, 4, 5, 6$ coincides with the alternating terms 
(those indexed by even $N = 2p$) of the OEIS sequence \href{https://oeis.org/A147681}{A147681}\cite{oeisA147681}, which counts `late-growing permutations' of $\{1, \dots, N\}$. A permutation $\sigma \in S_n$ is late-growing if its partial sums never exceed the corresponding partial sums of the average permutation, that is,
\[
    \frac1k\sum_{i=1}^{k} \sigma(i) \;\leq\; \frac{(n+1)}{2},
    \qquad \text{for all } k \leq n.\footnote{The right-hand side comes from the average of the numbers \(1,\dots,n\), which is \(\frac{n+1}{2}\). Thus the `average permutation' has partial sums \(\sum_{i=1}^k \frac{n+1}{2} = k\frac{n+1}{2}\), and dividing by \(k\) gives the stated bound.
}
\]
Our integer sequence \(\mathrm{const}(p)\), \(p\in\mathbb{N}_{\geq1}\) is now recorded in the OEIS as \href{https://oeis.org/A392714}{A392714}\cite{oeisA392714}.

\begin{claim}[O.~Zaboronski, private communication, June 2026]
\label{claim:positivity}
\[
\mathrm{const}(p) \;>\; 0.
\]
\end{claim}

\section{Growth rate of \texorpdfstring{$\mathrm{const}(p)$}{const(p)}}
\label{sec:growth}
\begin{remark}
    The values of our fast growing sequence \(\operatorname{const}(p)\) are accumulated by `integrating' over the permutation groups \(S_N\) with \(N=2p\). As \(p\to\infty\), the `volume' \(|S_N|=N!\) also grows super-exponentially. Suppose that the alternation (that yields \(\mathrm{const}(p)\)) is combined with the normalization by the size of the groups. In Table~\ref{tab:normal} we report the ratios $\mathrm{const}(p)/p!$ and $\mathrm{const}(p)/N!$ with $N=2p$. Neither ratio remains bounded: after an initial dip, both grow rapidly, indicating that $\mathrm{const}(p)$ eventually outpaces even $N!$ in growth rate.
\end{remark}
\begin{table}[H]
\centering
\caption{The sequence \(\mathrm{const}(p)\) grows super-exponentially, but so do the factorials in \(|S_N|=N!\) with \(N=2p\); here is their ratio.}
\label{tab:normal}
\renewcommand{\arraystretch}{1.6}
\begin{tabular}{crrrrr}
\hline\hline
\(p\)& \(p!\) & \(N!\) & \(\operatorname{const}(p)\)& \({\operatorname{const}(p)}/{p!}\) & \({\operatorname{const}(p)}/{N!}\) \\
\hline
$1$ & $1$& $2$ & $1$ & $1$& $0.5$   \\
$2$ & $2$& $24$& $2$ & $1$& $0.083$  \\
$3$ & $6$& $720$      & $90$& $15$& $0.125$   \\
$4$ & $24$      & $40\,320$  & $586\,656$     & $24\,444$  & $14.55$ \\
$5$ & $120$     & $3\,628\,800$ & $\approx 1.9\cdot 10^{12}$& $\approx 1.6\cdot 10^{10}$   & $\approx 5.2\cdot 10^{5}$   \\
$6$ & $720$     & $479\,001\,600$& $\approx 7.9\cdot 10^{21}$& $\approx 1.1\cdot 10^{19}$   & $\approx 1.6\cdot 10^{13}$  \\
\hline\hline
\end{tabular}
\end{table}

\subsection{Growth-rate fits for \texorpdfstring{$\mathrm{const}(p)$}{const(p)}}
\label{sec:growthfits}
\begin{claim}[O.~Zaboronski, private communication, June 2026]
\label{claim:zaboronski}
\[
p^2\log p\cdot\bigl(1+O(1/\log p)\bigr)\;\leqslant\;\log \mathrm{const}(p)\;\leqslant\;2\,p^2\log p\cdot\bigl(1+O(1/\log p)\bigr).
\]
\end{claim}

\noindent The ratios in Table~\ref{tab:normal} suggest that $\log(\mathrm{const}(p))$ grows like a low-degree polynomial in $p$ times a factorial-type logarithm. We test this by fitting $\log(\mathrm{const}(p))$ for $p=5,\dots,14$ (the values $p\le4$ are shown but excluded from the fit, being too small to constrain the asymptotics) against several two-parameter candidate laws by nonlinear least squares. Of seven candidate laws tested, the two best-fitting, both attaining $R^2>0.9999$, are
\[
\log(\mathrm{const}(p)) \approx c\,p\log(p!) + d\log((2p)!), \quad c = 1.76110,\ d = -0.98956, \quad R^2 = 0.999989,
\]
and
\[
\log(\mathrm{const}(p)) \approx c\,p^2\log p + d\,p^2, \qquad c = 1.71913,\ d = -1.71332, \qquad R^2 = 0.999967.
\]
The other five candidate laws we tried perform markedly worse; Table~\ref{tab:growthfits} ranks all seven by their maximal residual $\max|\mathrm{resid}|$ in $\log(\mathrm{const}(p))$-space.
\begin{table}[H]
\centering
\small
\caption{All seven candidate growth laws fitted to $\log(\mathrm{const}(p))$ for $p=5,\dots,14$ by nonlinear least squares, ranked by $\max|\mathrm{resid}|$ (best first).}
\label{tab:growthfits}
\renewcommand{\arraystretch}{1.3}
\begin{tabular}{l l c c}
\hline\hline
Model & Fitted parameters & $R^2$ & $\max|\mathrm{resid}|$ \\
\hline
$c\,p\log(p!)+d\log((2p)!)$ & $c=1.76110,\ d=-0.98956$ & $0.999989$ & $1.0713$ \\
$c\,p^2\log p+d\,p^2$ & $c=1.71913,\ d=-1.71332$ & $0.999967$ & $1.9428$ \\
$c\,(p^2\log p - p^2)$ & $c=1.72301$ & $0.999967$ & $2.0291$ \\
$c\,p\log(p!)+d\,p$ & $c=1.68354,\ d=-2.92665$ & $0.999934$ & $2.6143$ \\
$c\,p\log(p!)$ & $c=1.53680$ & $0.997127$ & $12.6283$ \\
$\alpha\log((2p)!)+b\,p$ & $\alpha=20.93663,\ b=-64.11317$ & $0.987221$ & $32.6112$ \\
$\alpha\log((2p)!)$ & $\alpha=6.43314$ & $0.831388$ & $117.8771$ \\
\hline\hline
\end{tabular}
\end{table}
Figure~\ref{fig:growthfits} shows both fitted curves together with the data.
\begin{figure}[H]
\centering
\includegraphics[width=0.75\linewidth]{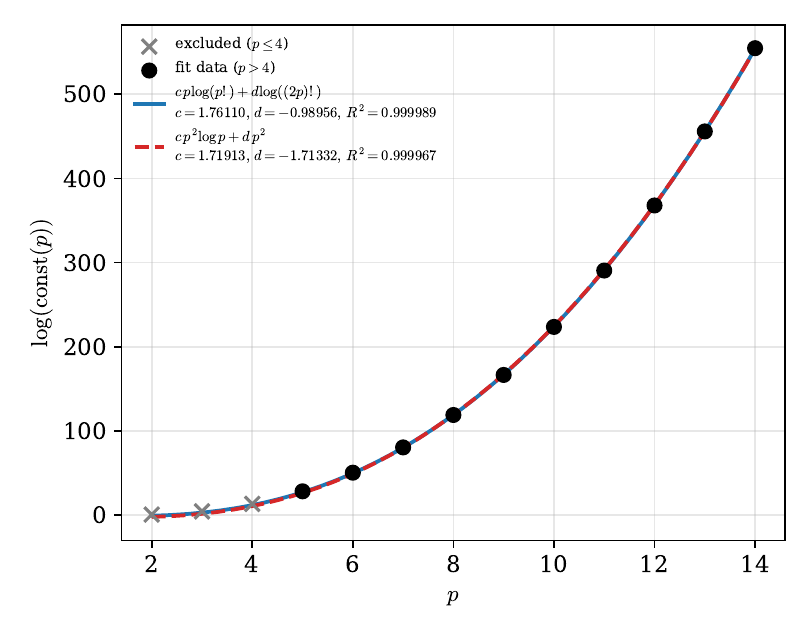}
\caption{The two best-fitting growth laws for $\log(\mathrm{const}(p))$ against $p$ (fit over $p>4$; the excluded points $p\le4$ are marked with a cross).}
\label{fig:growthfits}
\end{figure}

\medskip
\noindent Motivated by Claim~\ref{claim:zaboronski} and the two best fits above, we further fit two more general ansatz to $\log(\mathrm{const}(p))$ over the same range $p=5,\dots,14$: the full six-parameter `raw' polynomial in the primitive arguments $p$ and $\log p$,
\[
\log(c(p)) \approx \alpha\,p^2\log p+\beta\,p^2+\gamma\,p\log p+\delta\,p+\epsilon\log p+\mu,
\]
and the four-parameter factorial ansatz
\[
\log(c(p)) \approx A\,p\log(p!)+B\log((2p)!)+E\log(p!)+F.
\]
Ordinary least squares gives
\[
\alpha=1.998535,\ \beta=-2.667511,\ \gamma=0.909466,\ \delta=-0.027817,\]\[
\epsilon=3.023176,\ \mu=2.510459,\  \max|\mathrm{resid}|=9.7\times10^{-7},
\]
and
\[
A=1.827645,\ B=1.176431,\ E=-6.693739,\ F=-1.195608, \ \max|\mathrm{resid}|=6.0\times10^{-3}.
\]
Both leading coefficients, $\alpha\approx1.9985$ and $A\approx1.8276$, lie inside the interval $[1,2]$ delimited by Claim~\ref{claim:zaboronski}, with $\alpha$ in particular very close to the upper endpoint. Figure~\ref{fig:growthfitsfull} shows these two fits together with the data.
\begin{figure}[H]
\centering
\includegraphics[width=0.75\linewidth]{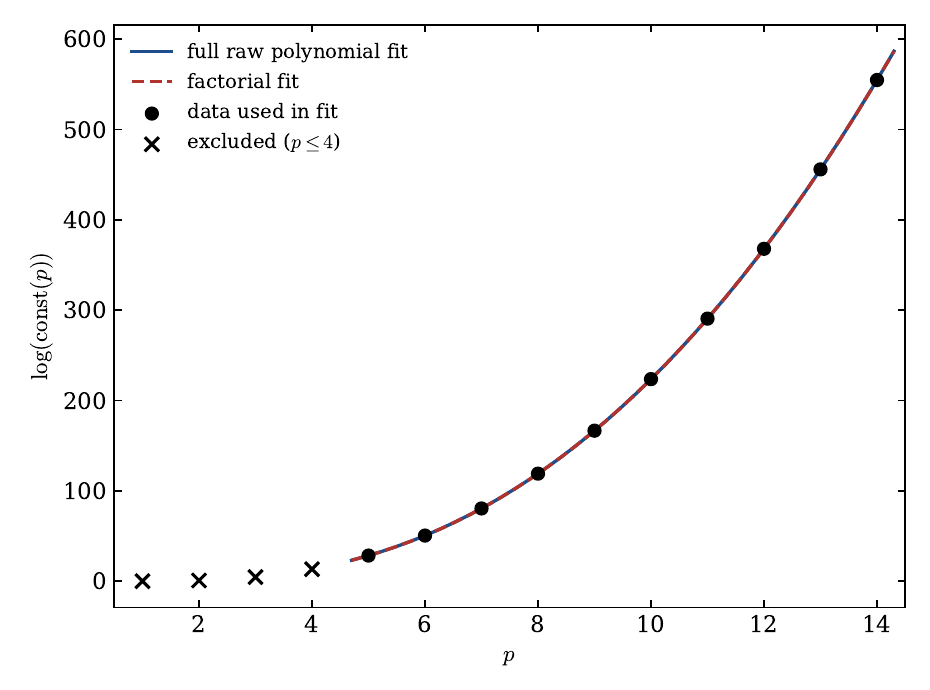}
\caption{The full raw-polynomial and factorial fits of $\log(\mathrm{const}(p))$ against $p$ (fit over $p>4$; the excluded points $p\le4$ are marked with a cross).}
\label{fig:growthfitsfull}
\end{figure}

\section{$p$-adic valuations of \texorpdfstring{$\mathrm{const}(p)$}{const(p)}}
\label{sec:padic}

\noindent We now shift our focus to some number-theoretic properties of our constants. We record the prime factorisations of our constants in Table~\ref{tab:factorization} (full table in Appendix~\ref{sec:allvalues}), and in
Table~\ref{tab:valuations} we display the \(p-\)adic valuation $v_q(\mathrm{const}(p))$ for all primes
$q \leq 19$ and $p \leq 14$, where a dot denotes a zero entry.

\begin{table}[h]
\centering
\small
\caption{Prime factorisation of $\mathrm{const}(p)$ for $2 \leq p \leq 12$.}
\label{tab:factorization}
\begin{tabular}{r@{\quad}p{0.75\textwidth}}
\toprule
$p$ & Prime factorisation of $\mathrm{const}(p)$ \\
\midrule
 2 & $2$ \\[2pt]
 3 & $2 \cdot 3^{2} \cdot 5$ \\[2pt]
 4 & $2^{5} \cdot 3^{3} \cdot 7 \cdot 97$ \\[2pt]
 5 & $2^{2} \cdot 3 \cdot 5^{4} \cdot 7 \cdot 79 \cdot 103 \cdot 4483$ \\[2pt]
 6 & $2^{5} \cdot 3^{6} \cdot 5^{2} \cdot 11 \cdot 23 \cdot 223 \cdot 239
     \cdot 1002820739$ \\[2pt]
 7 & $2^{4} \cdot 3^{2} \cdot 7^{6} \cdot 13 \cdot 29 \cdot 1361
     \cdot 2928279181 \cdot 3373618837829$ \\[2pt]
 8 & $2^{17} \cdot 3^{2} \cdot 7^{5} \cdot 13 \cdot 269 \cdot 9661 \cdot 4772558055911\cdot\; 1437228585673967071537$ \\[2pt]
 9 & $2^{11} \cdot 3^{14} \cdot 7^{2} \cdot 11^{2} \cdot 13
     \cdot 17 \cdot 43 \cdot 151 2 536393 \cdot 1193 919540 273283$\newline $\cdot 8 160866 422780 307366 231285 008303$\\[2pt]
10 & $2^{16} \cdot 3^{6} \cdot 5^{8} \cdot 19\cdot 2318 471913 717203$\newline $\cdot 15 042737 423986 318201 449432 278945 171198 976057 429145 160723 825671 082747$\\[2pt]
11 & $2^{7} \cdot 3^{5} \cdot 5^{6} \cdot 11^{10} \cdot 13 \cdot 19\cdots$ \\[2pt]
12 & $2^{18} \cdot 3^{11} \cdot 5^{5} \cdot 11^{8} \cdot 23\cdots$ \\[2pt]
13 & $2^{13}\cdot3^8\cdot5^2\cdot11^6\cdot13^{12}\cdot17\cdots$\\[2pt]
14&$2^{17}\cdot3^3\cdot5^2\cdot7^{13}\cdot11^4\cdot13^{11}\cdots$\\
\bottomrule
\end{tabular}
\end{table}

\begin{table}[h]
\centering
\small
\caption{$p$-adic valuations $v_q\bigl(\mathrm{const}(p)\bigr)$ for primes
$q \leq 19$ and $2 \leq p \leq 14$. The entries $v_p(\mathrm{const}(p))$
for prime $p$ are \textbf{bold}.}
\label{tab:valuations}
\begin{tabular}{r*{13}{c}}
\toprule
$q \,\backslash\, p$
  & 2 & 3 & 4 & 5 & 6 & 7 & 8 & 9 & 10 & 11 & 12&13&14 \\
\midrule
 2 & $\mathbf{1}$ &  1 &  5 &  2 &  5 &  4 & 17 & 11 & 16 &  7 & 18 &13&17\\
 3 & $\cdot$ & $\mathbf{2}$ &  3 &  1 &  6 &  2 &  2 & 14 &  6 &  5 & 11 &8&3 \\
 5 & $\cdot$ &  1 & $\cdot$ & $\mathbf{4}$ &  2 & $\cdot$ & $\cdot$ & $\cdot$ &  8 &  6 &  5 &2&2\\
 7 & $\cdot$ & $\cdot$ &  1 &  1 & $\cdot$ & $\mathbf{6}$ &  5 &  2 & $\cdot$ & $\cdot$ & $\cdot$ & $\cdot$&13 \\
11 & $\cdot$ & $\cdot$ & $\cdot$ & $\cdot$ &  1 & $\cdot$ & $\cdot$ &  2 & $\cdot$ & $\mathbf{10}$ &  8 &6&4\\
13 & $\cdot$ & $\cdot$ & $\cdot$ & $\cdot$ & $\cdot$ &  1 &  1 &  1 & $\cdot$ &  1 & $\cdot$ &$\mathbf{12}$&11\\
17 & $\cdot$ & $\cdot$ & $\cdot$ & $\cdot$ & $\cdot$ & $\cdot$ & $\cdot$&1 & $\cdot$ & $\cdot$ & $\cdot$&1 & $\cdot$\\
19 & $\cdot$ & $\cdot$ & $\cdot$ & $\cdot$ & $\cdot$ & $\cdot$ & $\cdot$ & $\cdot$&1&1 & $\cdot$ & $\cdot$ & $\cdot$\\
\bottomrule
\end{tabular}
\end{table}

\subsection{A proven lower bound on \texorpdfstring{$v_p(\mathrm{const}(p))$}{v\_p(const(p))}}
\label{sec:lowerbound}

\noindent In this section we will prove that for prime \(p\), \(v_p(\mathrm{const}(p))\geq p-1\), which gives a lower bound on $v_p(\mathrm{const}(p))$. 

\begin{notation}\label{not:padic}
For a prime $p$ and a nonzero rational number $r$, let $v_p(r)\in\mathbb{Z}$ denote its $p$-adic valuation: writing $r=p^{e}\cdot\frac{a}{b}$ with $a,b$ integers coprime
to $p$, we set $v_p(r):=e$. We recall the ultrametric inequality
\[
v_p(x+y)\;\geq\;\min\bigl(v_p(x),v_p(y)\bigr),
\]
valid for all rationals $x,y$; it extends inductively to any
finite sum\footnote{
\(
v_p\!\left(\sum_{i=1}^{n} x_i\right)\ge \min_{1\le i\le n} v_p(x_i)
\)
for all finite collections \(\{x_i\}\). Indeed, the case \(n=2\) is the ultrametric inequality. If the result holds for \(n-1\) terms, then
\[
v_p\!\left(\sum_{i=1}^{n}x_i\right)=
v_p\!\left(\left(\sum_{i=1}^{n-1}x_i\right)+x_n\right)
\ge
\min\!\left(
v_p\!\left(\sum_{i=1}^{n-1}x_i\right),
v_p(x_n)
\right)
\ge
\min_{1\le i\le n} v_p(x_i),
\]
which completes the induction.
}.
\end{notation}

\begin{lemma}\label{lem:masterclosed}
For $N=2p$ and $k=1,\dots,N-1$ the `master vector' (from Section~\ref{sec:deviation}) $v^{\mathrm m}_k$ equals
\[
v^{\mathrm m}_k \;=\; \frac{(k+1)(2p-k)}{2}.
\]
Consequently, over $k\in\{1,\dots,N-1\}$, $v^{\mathrm m}_k$ attains its maximum at
$k=p-1$ and at $k=p$, where
\[
\max_{1\le k\le N-1} v^{\mathrm m}_k \;=\; \frac{p(p+1)}{2}.
\]
\end{lemma}

\begin{proof}
Substituting $N=2p$ into $v^{\mathrm m}_k = k(N-1)-\tfrac{k(k-1)}{2}-(k-1)p$ and
simplifying gives
\[
v^{\mathrm m}_k = k(2p-1)-\frac{k(k-1)}{2}-(k-1)p
= p(k+1) - \frac{k(k+1)}{2} = \frac{(k+1)(2p-k)}{2}.
\]
As a function of the variable $k$, $(k+1)(2p-k)$ is a downward-opening parabola with vertex at $k=\tfrac{2p-1}{2}=p-\tfrac12$. Since this vertex is equidistant from the integers $p-1$ and $p$, both give the same maximal value $(k+1)(2p-k)\big|_{k=p-1}=(k+1)(2p-k)\big|_{k=p}=p(p+1)$, which yields the stated maximum of $v^{\mathrm m}_k$.
\end{proof}

\begin{lemma}\label{lem:Erange}
For every $\sigma\in\Phi_p$ and every $k=1,\dots,N-1$,
\[
p \;\le\; E_k(\sigma) \;\le\; \frac{p(p+1)}{2} \;<\; p^2.
\]
\end{lemma}

\begin{proof}
The lower bound is precisely the defining condition of $\Phi_p$ (\S\ref{sec:algorithm}); equivalently, $\delta_k(\sigma)\le v^{\mathrm m}_k-p$, as noted following
\eqref{eq:deviation}. The upper bound follows from $\delta_k(\sigma)\ge 0$ together with Lemma~\ref{lem:masterclosed}: $E_k(\sigma)=v^{\mathrm m}_k-\delta_k(\sigma)\le
v^{\mathrm m}_k\le\frac{p(p+1)}{2}$. The strict inequality $\frac{p(p+1)}{2}<p^2$ is equivalent to $p>1$, which holds for every prime $p$.
\end{proof}

\begin{lemma}\label{lem:fallingvaluation}
Let $p$ be a prime and let $m$ be an integer with $p\le m<p^2$. Then
$v_p\bigl(m!/(m-p)!\bigr) = 1$.
\end{lemma}

\begin{proof}
We recall Legendre's formula, which states that $v_p(n!)=\sum_{i\ge1}\lfloor n/p^i\rfloor$ for $n\ge0$. Both $m$
and $m-p$ lie in $[0,p^2)$, so every term with $i\ge2$ vanishes, giving
$v_p(m!)=\lfloor m/p\rfloor$ and $v_p((m-p)!)=\lfloor (m-p)/p\rfloor=\lfloor m/p\rfloor-1$.
Subtracting the two expressions gives $v_p\bigl(m!/(m-p)!\bigr)=v_p(m!)-v_p((m-p)!)=1$,
as claimed.
\end{proof}

\begin{lemma}\label{lem:cvaluation}
For every prime $p$ and every $\sigma\in\Phi_p$,
\[
v_p\bigl(c(\sigma,p)\bigr) \;=\; N-1 \;=\; 2p-1,
\qquad\text{where } c(\sigma,p):=\prod_{k=1}^{N-1}E_k^{\,\underline{p}}.
\]
\end{lemma}

\begin{proof}
By Lemma~\ref{lem:Erange}, each $E_k(\sigma)$ satisfies $p\le E_k(\sigma)<p^2$, so
Lemma~\ref{lem:fallingvaluation} gives $v_p\bigl(E_k(\sigma)^{\,\underline{p}}\bigr)=1$
for every $k=1,\dots,N-1$. Since $v_p$ is additive over products, $v_p(c(\sigma,p)) =
\sum_{k=1}^{N-1}v_p\bigl(E_k(\sigma)^{\,\underline{p}}\bigr) = N-1$, as claimed.
\end{proof}

\begin{lemma}\label{lem:Wvaluation}
Let $W:=\prod_{k=0}^{N-1}k!=\mathrm{Wronsk}(1,x,\dots,x^{N-1})$ (Lemma~\ref{lem:wronskian}).
Then $v_p(W) = p$.
\end{lemma}

\begin{proof}
For $0\le k\le p-1$, $v_p(k!)=0$ since $k<p$. For $p\le k\le N-1=2p-1$, we have $k<p^2$
(as $2p-1<p^2\iff(p-1)^2>0$, true for $p\ge2$), so Legendre's formula gives
$v_p(k!)=\lfloor k/p\rfloor=1$. Exactly $p$ of the $N=2p$ values of $k$ fall in this
upper range, each contributing $1$ and the rest contributing $0$, so $v_p(W)=p$.
\end{proof}

\begin{theorem}\label{thm:lowerbound}
For every prime $p$,
\[
v_p\bigl(\mathrm{const}(p)\bigr) \;\geq\; p-1.
\]
\end{theorem}

\begin{proof}
By Claim~\ref{claim:closedform},
\[
\mathrm{const}(p)\cdot W \;=\; \sum_{\sigma\in\Phi_p}(-1)^\sigma\,c(\sigma,p).
\]
By Lemma~\ref{lem:cvaluation}, every summand on the right satisfies
$v_p\bigl((-1)^\sigma c(\sigma,p)\bigr)=2p-1$, the sign not affecting the valuation.
By the ultrametric inequality (Notation~\ref{not:padic}),
\[
v_p\!\left(\sum_{\sigma\in\Phi_p}(-1)^\sigma\,c(\sigma,p)\right)\;\geq\;
\min_{\sigma\in\Phi_p} \left\{v_p\bigl((-1)^\sigma c(\sigma,p)\bigr)\right\} \;=\; 2p-1.
\]
Thus $v_p(\mathrm{const}(p)\cdot W)\ge 2p-1$, and by Lemma~\ref{lem:Wvaluation},
$v_p(\mathrm{const}(p)) = v_p(\mathrm{const}(p)\cdot W) - v_p(W) \ge (2p-1)-p = p-1$,
as claimed.
\end{proof}

One feature of Table~\ref{tab:valuations} is immediately apparent: the prime factors of $p$ appear in $\mathrm{const}(p)$ with notably high valuation. The dominant contributions to $v_q(\mathrm{const}(p))$ arise precisely when $q \mid p$. Particularly, for every prime $p$ in our range,
the entry $v_p(\mathrm{const}(p))$ equals exactly $p - 1$.

\begin{conjecture}\label{conj:valuation}
For every prime $p$,
\[
    v_p\bigl(\mathrm{const}(p)\bigr) \;=\; p - 1.
\]
\end{conjecture}

\begin{remark}
Theorem~\ref{thm:lowerbound} establishes half of Conjecture~\ref{conj:valuation}
unconditionally for every prime $p$. Writing $\mathrm{const}(p)=p^{p-1}\cdot S_p$ for the (now
guaranteed) integer $S_p:=\mathrm{const}(p)/p^{p-1}$, the remaining content of the
conjecture is the single statement $p\nmid S_p$, i.e.\ the absence of any further
cancellation modulo $p$ in the alternating sum of Claim~\ref{claim:closedform} beyond
what is accounted for above. 
\end{remark}

\section{The constants \texorpdfstring{$D(p)$ and $\overline D(p)$}{D(p) and Dbar(p)}}
\label{sec:Dp}

\noindent Recall the Wronskian factor $W=\prod_{k=1}^{N-1}k!=\prod_{k=0}^{N-1}k!$ (Lemma~\ref{lem:wronskian}), and set
\[
D(p) \;:=\; \frac{c(p)\cdot W}{(p!)^{N-1}}, \qquad \overline D(p) \;:=\; \frac{D(p)}{W} \;=\; \frac{c(p)}{(p!)^{N-1}}.
\]
Direct computation from the exact values of $c(p)=\mathrm{const}(p)$ shows $D(p)$ is a positive integer for every $p\le14$ (full list in Appendix~\ref{app:Dvalues}), growing to a 269-digit integer at $p=14$. In fact $D(p)>c(p)$ throughout this range, with the gap widening rapidly ($D(14)/c(14)\approx3.2\times10^{27}$) - this is a direct consequence of $W$ growing faster than $(p!)^{N-1}$.

The reduced constant $\overline D(p)$ decays to $0$ super-exponentially ($\overline D(1)=1,\ \overline D(2)=1/4,\ \ldots,\ \overline D(14)\approx3.01\times10^{-55}$). Figure~\ref{fig:Dpfits} shows $\log D(p)$ alongside $\log c(p)$ for reference, and the growth of $-\log\overline D(p)$ together with its factorial fit ($-\log\overline D(p)\approx0.1724\,p\log(p!)-1.1764\log((2p)!)+5.6937\log(p!)+1.1956$, $\max|\mathrm{resid}|=6.0\times10^{-3}$).
\begin{figure}[H]
\centering
\includegraphics[width=\linewidth]{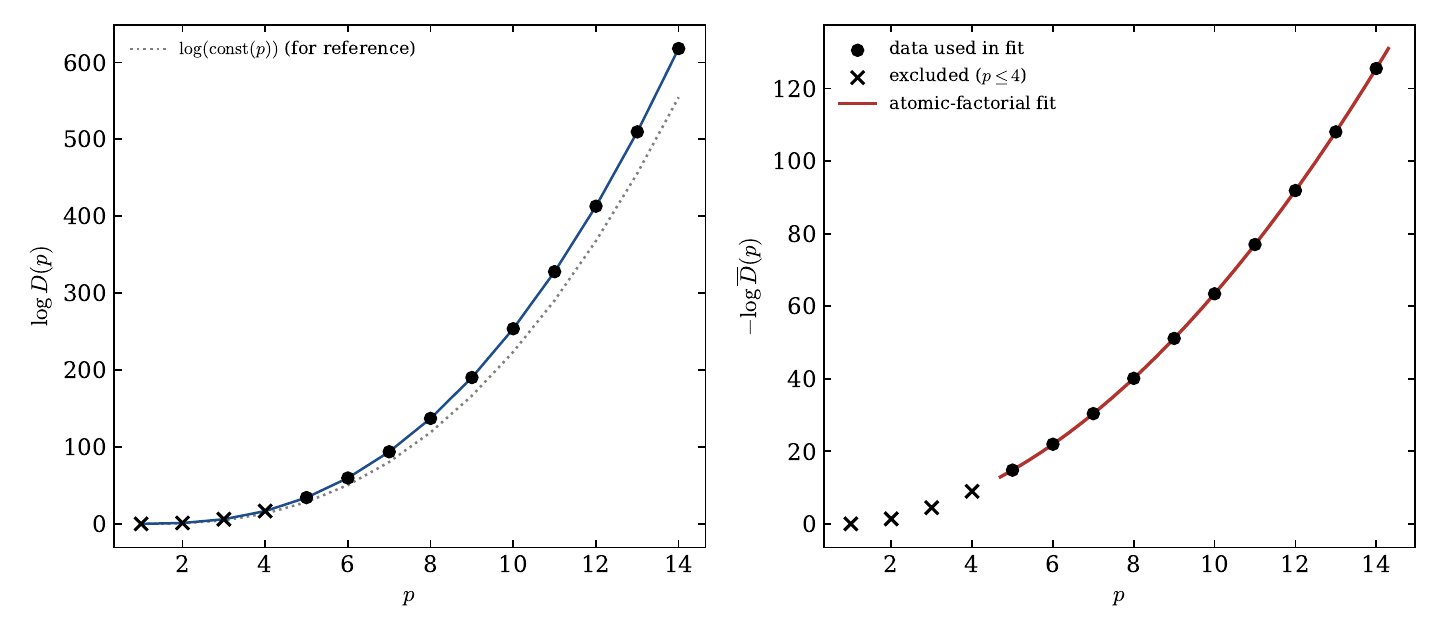}
\caption{Left: $\log D(p)$ (solid) compared with $\log(\mathrm{const}(p))$ (dotted) for reference. Right: $-\log\overline D(p)$ together with its factorial least-squares fit; both confirm that $\overline D(p)\to0$ super-exponentially. Excluded points $p\le4$ are marked with a cross.}
\label{fig:Dpfits}
\end{figure}

\section*{Acknowledgements}
\noindent We thank R.~Buring and O.~Zaboronski for fruitful collaboration and discussions. The first author thanks the Center for Information Technology of the University of Groningen for access to the High Performance Computing cluster, H\'abr\'ok, and 
K.~C.\ Waggeveld for performing the computations thereon. The second author thanks the organisers of the Prague Mathematical Physics Seminar (at the Charles University), where this result was presented on 7 May 2026, for helpful discussion, 
partial financial support, and hospitality. The second author thanks M.~Kontsevich and the audience of the Mathematics seminar at the IH\'ES (13 May 2026, Bures\/-\/sur\/-\/Yvette, France) for stimulating discussions. The first author thanks the organisers of \textsc{Group36} workshop on group theoretic methods in Physics (13--17 July 2026, Valladolid, Spain) for a warm atmosphere during the meeting. This work has been partially supported by 
the Bernoulli Institute (Groningen, NL) via project 110135.

\printbibliography
\appendix
\section{Proof of Theorem \ref{thm:main}}
\label{appendixa}
The proof relies on the following two lemmas and the corollary
they imply.
\begin{lemma}[Antisymmetry]\label{lem:antisym}
The alternating composition satisfies, for any $a \neq b$,
\[
  \mathcal{A}_p[\dots, w_a, \dots, w_b, \dots](f)
  \;=\;
  -\,\mathcal{A}_p[\dots, w_b, \dots, w_a, \dots](f).
\]
\end{lemma}

\begin{proof}
Compose each summation index $\sigma$ in~\eqref{eq:opdef} with the
transposition $\tau = (a\ b)$.
As $\sigma$ ranges over all of $S_N$, so does $\sigma\circ\tau$,
and $(-1)^{\sigma\circ\tau} = -(-1)^\sigma$.
Simultaneously, composing with $\tau$ interchanges the labels $a$
and $b$ throughout, which is equivalent to swapping $w_a$ and $w_b$
in the weight list.
Relabelling $\sigma \mapsto \sigma\circ\tau$ in the sum thus
changes its sign, proving the claim.
\end{proof}

\begin{lemma}[Derivative count]\label{lem:count}
In every monomial produced by the full Leibniz expansion of any
summand of $\mathcal{A}_p[w_1,\dots,w_N](f)$, the total number of
$x$-derivatives is exactly
\[
  N \cdot p \;=\; 2p^2.
\]
\end{lemma}

\begin{proof}
Each summand in~\eqref{eq:opdef} is the composition of $N = 2p$
operators, each contributing exactly $p$ applications of $\partial_x$.
The Leibniz rule redistributes these derivatives among the factors in every term, but preserves the total count. Thus every monomial in the expansion carries exactly $Np = 2p^2$ derivatives.
\end{proof}

\begin{corollary}[Order of $\mathcal{A}_p$]\label{cor:orderp}
The operator $\mathcal{A}_p[w_1,\dots,w_N]$ has differential order exactly
$p$ in $f$.
That is, every surviving term in the expansion of~\eqref{eq:opdef}
carries the factor $\partial_x^p(f)$, and no term $\partial_x^k(f)$
with $k \neq p$ appears with non-zero coefficient.
\end{corollary}

\begin{proof}
Let $k$ denote the differential order of $f$ in a given term of the
full Leibniz expansion; so the term contains the factor $\partial_x^k(f)$.

\smallskip
\noindent\emph{Lower bound $k \geq p$.}
In each summand of~\eqref{eq:opdef}, the innermost operation is
$w_{\sigma(N)}\partial_x^p(f)$, which already contains $\partial_x^p(f)$.
Every subsequent application of $\partial_x^p$ from the left acts on
an expression involving $\partial_x^p(f)$ via the Leibniz rule:
for any smooth $g$ and $m \geq 0$,
\[
  \partial_x^m\!\left(g\cdot\partial_x^p(f)\right)
  = \sum_{l=0}^{m}\binom{m}{l}\partial_x^l(g)\cdot\partial_x^{p+m-l}(f),
\]
so the order of $f$ in each resulting term is $p + m - l \geq p$.
Proceeding inductively through all $N$ composition steps,
every monomial satisfies $k \geq p$.

\smallskip
\noindent\emph{Upper bound $k \leq p$.}
We can write the $2p^2$ derivatives of Lemma~\ref{lem:count} as $k$ acting
on $f$ and $2p^2 - k$ acting on the weight functions, with
non-negative orders $k_1, \dots, k_N$ ($k_i$ derivatives on $w_i$) satisfying
$\sum_{i=1}^N k_i = 2p^2 - k$.

By Lemma~\ref{lem:antisym}, if $k_a = k_b$ for some $a \neq b$, then swapping $w_a$ and $w_b$
maps the monomial to itself while negating its coefficient, forcing
that coefficient to zero. Thus every term with non-zero coefficient
has \emph{pairwise distinct} values $k_1, \dots, k_N$.

Since $k_1, \dots, k_N$ are $N = 2p$ pairwise distinct non-negative
integers,
\begin{equation}\label{eq:lb}
  \sum_{i=1}^{N} k_i
  \;\geq\;
  0 + 1 + 2 + \cdots + (N-1)
  \;=\;
  \frac{N(N-1)}{2}
  \;=\;
  p(2p-1).
\end{equation}
Thus
\[
  2p^2 - k \;=\; \sum_{i=1}^{N}k_i \;\geq\; p(2p-1) = 2p^2 - p,
\]
which gives $k \leq p$.

Combining the two bounds yields $k = p$.
\end{proof}

\begin{remark}\label{rem:budget}
The lower bound in~\eqref{eq:lb} is tight: equality holds if and only
if $\{k_1, \dots, k_N\} = \{0, 1, \dots, N-1\}$.
The budget for weight derivatives is exactly $2p^2 - p = p(2p-1)$,
matching this minimum precisely.
Any other profile of pairwise distinct non-negative integers
- for instance $\{0,1,2,3,4,6\}$ when $N = 6$, $p = 3$, requiring
$0+1+2+3+4+6 = 16 > 3 \cdot 5 = 15$ derivatives - exceeds the
available budget and cannot contribute.
\end{remark}

\begin{proof}[Proof of Theorem~\ref{thm:main}]
By Corollary~\ref{cor:orderp}, every surviving term in the expansion of\\ $\mathcal{A}_p[w_1,\dots,w_N](f)$ has the form
\[
  C_{\mathbf{k}}\cdot\prod_{i=1}^{N}\partial_x^{k_i}(w_i)\cdot\partial_x^p(f),
\]
where $C_{\mathbf{k}}$ is a numerical coefficient and
$\mathbf{k} = (k_1, \dots, k_N)$ is the order profile on the weights.

As shown in the proof of Corollary~\ref{cor:orderp}, antisymmetry
forces each surviving profile to have pairwise distinct $k_i$.
With $k = p$ fixed, the derivative budget on the weights is
$2p^2 - p = p(2p-1)$, which by Remark~\ref{rem:budget} is
exactly the minimum achievable by $N$ pairwise distinct non-negative
integers. Thus equality must hold in~\eqref{eq:lb}, and the
\emph{unique} surviving profile is
\[
  \{k_1, \dots, k_N\} \;=\; \{0, 1, \dots, N-1\}.
\]
The alternating sum over all assignments of this profile to the weight
functions is, by definition, the Wronskian:
\[
  \Wronsk(w_1,\dots,w_N)
  \;:=\;
  \det\!\bigl[\partial_x^{j-1}(w_i)\bigr]_{1\leq i,j\leq N}
  \;=\;
  \sum_{\sigma\in S_N}(-1)^\sigma\prod_{i=1}^{N}\partial_x^{\sigma(i)-1}(w_i).
\]
The coefficient $C_{\mathbf{k}}$ arises purely from the combinatorics
of the iterated Leibniz expansion and depends only on $p$, not on $f$
or on the weight functions. Denoting this factor by
$\mathrm{const}(p)$, we conclude
\[
  \mathcal{A}_p[w_1,\dots,w_N](f)
  \;=\;
  \mathrm{const}(p)\cdot\Wronsk(w_1,\dots,w_N)\cdot\partial_x^p(f). 
\]
\end{proof}
\section{Values of $\mathrm{const}(p)$ up to $p=14$}
\label{sec:allvalues}
Here we present the known values of $\mathrm{const}(p)$ up to $p = 14$,
together with their prime factorizations.

\smallskip
\noindent The prime factorisations below were obtained with Dario Alpern's Integer Factorisation Calculator (ECM/SIQS)~\cite{alpertron}, which in our usage reliably extracts prime factors up to about 30 decimal digits. For each $\mathrm{const}(p)$, once no further factor below this threshold could be found, the remaining cofactor was not factored further; instead we subjected it to a probabilistic primality test using bigprimes.org~\cite{bigprimesorg}, which applies the Baillie-PSW test together with 7 rounds of the Miller-Rabin test for numbers exceeding 64 bits. A cofactor that passes this test is reported below as a single (unfactored) integer, and should be understood as only \emph{probably} prime rather than proven prime; we have not attempted a deterministic primality proof (e.g.\ by ECPP) for these cofactors. The sole exception is the large cofactor appearing in the factorisation of $\mathrm{const}(13)$, which fails the probabilistic test and is accordingly labelled `(Composite)' without being split further.

\bigskip
\noindent\rule{\linewidth}{0.4pt}
\smallskip
\small

\constentry{1}{2}{1}{1}

\constentry{2}{4}{2}{2}

\constentry{3}{6}{90}{$2\cdot3^2\cdot5$}

\constentry{4}{8}{586656}{$2^5\cdot 3^3\cdot 7\cdot 97$}

\constentry{5}{10}{1915103977500}{$2^2 \cdot3 \cdot 5^4 \cdot 7 \cdot 79 \cdot 103 \cdot 4483$}

\constentry{6}{12}{7886133184567796056800}{$2^5 \cdot 3^6 \cdot 5^2 \cdot 11 \cdot 23 \cdot 223 \cdot 239 \cdot 1002 820739$}

\constentry{7}{14}{85873408332103907284746052081828368}{ $2^4 \cdot 3^2 \cdot 7^6 \cdot 13 \cdot 29 \cdot 1361 \cdot 2928 279181 \cdot 3 373618 837829$}

\constentry{8}{16}{4594491123326092088701002220876785865521537220214784}{$2^{17} \cdot 3^2 \cdot 7^5 \cdot 13 \cdot 269 \cdot 9661 \cdot 4 772558 055911 \cdot 1437 228585 673967 071537$}

\constentry{9}{18}{2059560342372549855311520646192231857828658156995213197895311773186910208}{$2^{11} \cdot 3^{14} \cdot 7^2 \cdot 11^2 \cdot 13 \cdot 17 \cdot 43 \cdot 151 \cdot 2 536393 \cdot 1193 919540 273283 \cdot 8 160866 422780 307366 231285 008303$ (31 digits)}

\constentry{10}{20}{12366585616687922175229690518161006581890827598464384763435846689194555321926721635186969600000000}{$2^{16} \cdot 3^6 \cdot 5^8 \cdot 19 \cdot 2318 471913 717203 \cdot 3489 987569 680376 097511 $\\$\cdot 4310 255301 386067 349690 402472 139980 447858 870477$ (46 digits)}

\constentry{11}{22}{1512170566890529388470809077236613950878090470607103334738020704327297913201742894343506377843406342016226367608547680866000000}{ $2^7 \cdot 3^5 \cdot 5^6 \cdot 11^{10} \cdot 13 \cdot 19 \cdot 25457 \cdot 1 304111 \cdot 19 027007 \cdot 277 381877 879789 \cdot 2724 104292 276785 198723 \cdot 1017 528619 742366 754559 096610 489525 888417 939019 147331$ (52 digits)}

\constentry{12}{24}{5498204731735819824932764431577828853091535075963690476426743279778624624056195101600179165694626009665228817515586720358929229075113647839992833114429030400000}{$2^{18} \cdot 3^{11} \cdot 5^5 \cdot 11^8 \cdot 23 \cdot 6240 374221 \cdot 4 156745 629145 249267 \cdot 150087 192783 934271 953579 \cdot 1 973877 182896 509567 154686 476569 916604 851423 966617 859790 944299 941672 435839 359560 447739$ (85 digits)}

\constentry{13}{26}{839630392364690485033794364133039282555694250991715962816289209038314705544359344129440199511545622749161083016799037230605187931570735637773841615964174834585398693013887219584559285780585126502400}{$2^{13}\cdot3^8\cdot5^2\cdot11^6\cdot13^{12}\cdot17\cdot97\cdot337\cdot27243397385914871724447417691488851104\\57471322785266427208450071853005505003901740433232104186796429270064993602844733145\\8925288320887246456354056141933129121535401$ (164 digits) (Composite)}

\constentry{14}{28}{7398040169232585459350131400876616433108456563542326671820170467323437649015925253611993741824016395372908844522732645092417947941972247942351664804568139456872774721959807686464898490851546742468976729501727604943676101572611708283345305600}{$2^{17}\cdot3^3\cdot5^2\cdot7^{13}\cdot11^4\cdot13^{11}\cdot32891285673563672246003165608072510232318699\\77405135791308188235300055952985149062873180268302320181032892587194456510949295544\\6073147116264586137549690490817951980312345150293192647954551600837443915441259$ (206 digits)}

\normalsize

\section{Values of $D(p)$ up to $p=14$}
\label{app:Dvalues}

\noindent Here we tabulate the exact values of the auxiliary integer $D(p)$ (\S\ref{sec:Dp}), defined by $D(p) = \mathrm{const}(p)\cdot W/(p!)^{N-1}$ with $N=2p$ and $W=\prod_{k=1}^{N-1}k!$.

\bigskip
\noindent\rule{\linewidth}{0.4pt}
\smallskip
\small

\dentry{1}{2}{1}

\dentry{2}{4}{3}

\dentry{3}{6}{400}

\dentry{4}{8}{16041375}

\dentry{5}{10}{681053872728096}

\dentry{6}{12}{77758114387378156351627392}

\dentry{7}{14}{50282641110972993477580343756906229399552}

\dentry{8}{16}{343157806261282476591438787344855377532318432677301031527375}

\dentry{9}{18}{42223612607128022427120264390241431004508083719691021491054250153037118519280000000}

\dentry{10}{20}{149947853037782193993631030647104952918725205895602925886651767593663381180948805338711916412142453101148749824}

\dentry{11}{22}{23387688366898355405774901194352041197121294755161554731078563089425366689062928236233917164484632968652160269181879258218085056918584924569600}

\dentry{12}{24}{234092958692567435267567612995154846092832846049300128958763346868018256895439024455278179218144715860090755804299829586277766866586228653803534930390084262752539591620885158400000}

\dentry{13}{26}{212495240502868440365595901979832626678808135660756986257865427791856362728008911094411915209979327655795531969686388207399690548309749245632129410868653691092463268188091692016418350834282640515523187682743366451200000000}

\dentry{14}{28}{24041522289912242513820465051878116815132836110768798151462541478965187632793330706007079082062139885673821124016196800742177030484089078426807160325958273054580787456334494337568049521229191175845512069960827975261580611786245976237740777845381002690560000000000000000}

\normalsize

\end{document}